\documentclass[11pt,reqno]{amsart}
\usepackage{amssymb,latexsym}
\usepackage[mathscr]{eucal}
 \usepackage{hyperref, flafter,epsf}
\usepackage{amsmath, amsthm, amsfonts, amscd, amssymb}
\usepackage{latexsym}
\usepackage[dvips]{graphics}
 \usepackage{fancyhdr}
 \usepackage{times}
 \usepackage[T1]{fontenc}
\usepackage{epsfig}
\usepackage{color}
 \usepackage{mathrsfs}
\usepackage{verbatim}
\usepackage{epstopdf}
\input amssym.def
\input amssym.tex

\newtheorem{thm}{Theorem}
\newtheorem{lem}{Lemma}
\newtheorem{cor}{Corollary}
\newtheorem{con}{Conjecture}

\newcommand{\be}{\begin{equation}}
\newcommand{\ee}{\end{equation}}

 \usepackage{parskip}
 \usepackage{graphicx}
\usepackage{hyperref}

\usepackage[OT2,OT1]{fontenc}
\newcommand\cyr{%
\renewcommand\rmdefault{wncyr}%
\renewcommand\sfdefault{wncyss}%
\renewcommand\encodingdefault{OT2}%
\normalfont
\selectfont}
\DeclareTextFontCommand{\textcyr}{\cyr}

\title[Small zeros of Dirichlet $L$-functions of prime modulus]
{Small zeros of Dirichlet $L$-functions of quadratic characters of prime modulus}

\author{Julio Andrade}
\address{Department of Mathematics, University of Exeter, North Park Road,
Exeter EX4 4QF, UK}
\email{j.c.andrade@exeter.ac.uk}

\author{Siegfred Baluyot}
\address{Department of Mathematics, University of Rochester, Rochester, NY 14627, USA}
\curraddr{Department of Mathematics, University of Illinois, Urbana, IL 61801, USA}
\email{sbaluyot@illinois.edu}

\thanks{JCA would like to thank the Department of Mathematics at the University of Rochester for hospitality during a visit where this work started. JCA was partially supported by a Research in Pairs - Scheme 4 London Mathematical Society grant.}


\begin{document}

 \maketitle

\begin{abstract}

In this paper, we investigate the distribution of the imaginary parts of zeros
near the real axis of Dirichlet $L$-functions associated to the quadratic characters
$\chi_{p}(\cdot)=(\cdot |p)$ with $p$ a prime number. Assuming the Generalized
Riemann Hypothesis (GRH), we compute the one-level density for the zeros of this family of
$L$-functions under the condition that the Fourier transform of the test function is
supported on a closed subinterval of $(-1,1)$. We also write down the
ratios conjecture for this family of $L$-functions a la Conrey, Farmer and Zirnbauer
and derive a conjecture for the one-level density which is consistent with the main theorem
of this paper and with the Katz-Sarnak prediction and includes lower order terms. Following the methods
of \"{O}zl\"{u}k and Snyder, we prove that GRH implies $L(\frac{1}{2},\chi_p)\neq 0$ for at least $75\%$ of the primes.

\end{abstract}



\section{Introduction}

The distribution of the imaginary parts of zeros near the real axis
of $L$-functions is an important theme in analytic
number theory. The zeros of Dirichlet $L$-functions close to the real
axis encapsulate important information about several number theoretical
quantities. For example, the zeros of $L(s,\chi)$ near to $s=1/2$, with
$\chi$ a quadratic character such that $\chi(-1)=-1$, are related to
the class numbers of complex quadratic fields. In another direction, if
$L(s,\chi)$ is the Dirichlet $L$-function associated to the non-principal
character $\chi$ modulo $4$ then the first low zeros of $L(s,\chi)$ dictate
how the primes are distributed in residue classes $1$ and $3$ $\pmod 4$.
For further details on how low zeros of Dirichlet $L$-functions are related
to problems in number theory we refer the reader to \cite{BP, Ben, MW, Wei}.

\"{O}zl\"{u}k and Snyder \cite{ozluksnyder1,ozluksnyder2} have studied the distribution of low-lying
zeros of quadratic Dirichlet $L$-functions $L(s,(d|\cdot))$ for all fundamental
discriminants $d$ that have absolute value less than or equal to a given
constant $D$. Their approach is to investigate the asymptotic properties of the \textit{form factor}
$$
F(\alpha,D) \ = \ \left( \frac{1}{2\zeta(2)}K\left( \frac{1}{2}\right)D \right)^{-1}\sum_{d\in\mathcal{F}(D)}\sum_{\rho(d)}K(\rho)D^{i\alpha\gamma}
$$
as $D\rightarrow\infty$. Here, $K$ is a suitable kernel, the outer sum is over all fundamental
discriminants $d$ with $|d|\leq D$, and the inner sum is over the non-trivial zeros of $L(s,\chi_{d})$, where
$\chi_{d}(n)$ is defined by the Kronecker symbol $(d|n)$. Assuming the Generalized Riemann Hypothesis (GRH), \"{O}zl\"{u}k and Snyder~\cite[Theorem 3]{ozluksnyder1} established a formula for the one-level
density for the zeros of $L(s,\chi_{d})$, and improved their formula in a subsequent paper~\cite[Corollary 2]{ozluksnyder2}. Their results agree with those of Katz and Sarnak~\cite[page 16]{KS99}.

The formulas for the one-level density can be interpreted in the following form, as discussed in Entin, Roditty-Gershon, and Rudnick~\cite{ERR}. To simplify the discussion, we restrict to discriminants of the form $8d$, with $d>0$ an odd square-free integer. Thus the corresponding quadratic character $\chi_{8d}$ is primitive and even with conductor $8d$.
We can define the linear statistic of zeros of $L(s\chi_{8d})$, or one-level density, by taking $f$ to be an even Schwartz function
and setting
\begin{equation*}
W_{f}(d) \ = \ \sum_{\gamma}f\left(\frac{\gamma\log X}{2\pi}\right),
\end{equation*}
where the sum is over all the zeros $\frac{1}{2}+i\gamma$ of $L(s,\chi_{8d})$. Let $\mathcal{D}(X)$ denote the set of odd square-free $d$ in the interval $[X,2X]$. Katz and Sarnak~\cite{KS99} proved under the assumption of GRH that, in the limit as $X\rightarrow\infty$, the expected value of $W_f(d)$ over
the ensemble $\mathcal{D}(X)$ coincides with the analogous quantity for the eigenphases of random matrices from unitary symplectic groups $USp(2N)$ in the limit as
$N\rightarrow\infty$. In other words,
\begin{equation}\label{DensityConjecture}
\lim_{X\rightarrow\infty}\frac{1}{\#\mathcal{D}(X)} \sum_{d\in\mathcal{D}(X)}W_{f}(d)\Phi\left( \frac{d}{X}\right) \ = \ \int_{-\infty}^{\infty}f(x)\left( 1-\frac{\sin2\pi x}{2\pi x}\right)dx,
\end{equation}
under the restriction that the Fourier transform $\widehat{f}(u)=\int_{\mathbb{R}}f(x)e^{-2\pi ixu}dx$ is supported in the interval $|u|<2$, where $\Phi$ is a smooth weight function supported in the interval $(1,2)$ such that $\int\Phi(u)du=1$. An earlier version of this formula, in a different form, was proved by \"{O}zl\"{u}k and Snyder~\cite{ozluksnyder1,ozluksnyder2}. The \textit{Density Conjecture} of Katz and Sarnak~\cite{KS99} is the conjecture that \eqref{DensityConjecture} holds for any test function $f$, without restrictions on its Fourier transform. This is still intractable with the current technology but in the last few years striking progress was obtained in our understanding of the distribution of zeros of Dirichlet $L$-functions by several authors including Fiorilli and Miller \cite{FM}, Gao \cite {G}, Hughes and Rudnick \cite{HR}, Levinson and Miller \cite{LM}, Miller \cite{Mi1} and Rubinstein \cite{Rub} to quote some of them.

In this paper we study the one-level density for the family of quadratic Dirichlet $L$-functions associated to the characters $\chi_{p}$ defined by the Legendre symbol $(\cdot |p)$, where $p$ is a prime. One of the first authors to investigate this family of $L$-functions was Jutila in 1981 \cite{Jutila}. Jutila proved an asymptotic formula for the first moment at the central point for this family of $L$-functions. Studying this family of $L$-functions seems to be more difficult than studying the family of quadratic Dirichlet $L$-functions $L(s,\chi_{8d})$. One reason for this is that averages over prime numbers bring extra difficulties which we do not face when dealing with averages over square-free numbers. For example, asymptotic formulas for the first three moments of $L(\frac{1}{2},\chi_{8d})$ over $\mathcal{D}(X)$ are known, with the first two essentially due to Jutila~\cite{Jutila}, and the third due to Soundararajan~\cite{Sou}. However, an asymptotic formula is known only for the first moment of $L(\frac{1}{2},\chi_{p})$ over the primes $p\leq X$~\cite{Jutila}. In the function field setting, Andrade and Keating~\cite{AK} have established formulas for the first and second moments of the analogue of this family.

The main aim of this paper is to initiate an investigation of the statistical distribution of the zeros for the family of quadratic Dirichlet $L$-functions $L(s,\chi_{p})$ and to understand the main differences when dealing with $L(s,\chi_{p})$ instead of $L(s,\chi_{d})$. We assume the Generalized Riemann Hypothesis (GRH) all throughout. Our main result is analogous to the formula \eqref{DensityConjecture}, and holds for test functions $f$ with Fourier transform $\hat{f}(u)$ supported in the interval $|u|<1$. We follow the approach of \"{O}zl\"{u}k and Snyder~\cite{ozluksnyder1,ozluksnyder2} and develop a formula for the form factor for this family. Whereas \"{O}zl\"{u}k and Snyder average over all fundamental discriminants $d$, we average over the more restricted set of primes $p$. Averaging over all fundamental discriminants allows one to use the Poisson summation formula, as \"{O}zl\"{u}k and Snyder~\cite{ozluksnyder2} do. However, there does not seem to be an analogue for Poisson summation when averaging over the primes. Thus we resort to directly estimating the contribution of the off-diagonal terms by bounding the character sums, as in Lemma~\ref{GRHbound} in Section~\ref{offdiag} below. 
This bound, on GRH, allows us to prove an asymptotic formula for the one-level density of zeros when $\hat{f}(u)$ is supported in the interval $|u|<1$. It would be interesting to see how the support of $\hat{f}(u)$ can be increased to deduce results of the same quality as \"{O}zl\"{u}k and Snyder.

Aside from determining the one-level density of zeros for the family of $L(s,\chi_p)$, we also use our formula for the form factor to prove that, under GRH, $L(\frac{1}{2},\chi_{p})$ is nonzero for at least $75\%$ of the primes. The analogous conditional result for the family of $L(s,\chi_{d})$, due to \"{O}zl\"{u}k and Snyder~\cite{ozluksnyder2}, is that $L(\frac{1}{2},\chi_{d})$ is nonzero for at least $93.75\%$ of the fundamental discriminants $d$. Soundararajan~\cite{Sou} has proved unconditionally that $L(\frac{1}{2},\chi_{d})\neq 0$ for at least $87.5\%$ of the fundamental discriminants. As far as we know, the current best unconditional nonvanishing result for the family of $L(s,\chi_{p})$ is due to Jutila~\cite{Jutila}, who proved that $L(\frac{1}{2},\chi_{p})\neq 0$ for infinitely many primes $p\equiv v$ (mod~$4$) for each of $v=1$ and $v=3$.

Besides studying the distribution of zeros of the family of $L$-functions $L(s,\chi_{p})$, we also write down the ratios conjecture for this family, following the ideas of Conrey, Farmer, and Zirnbauer~\cite{CFZ} and Conrey and Snaith~\cite{CS}. Using the ratios conjecture, we infer a precise formula for the average one-level density, complete with lower order terms. The consequence of the ratios conjecture is consistent with the Katz-Sarnak prediction that the analogue of \eqref{DensityConjecture} holds for any test function $f$. In future work, we plan to develop the ideas of this paper further and study the $n$-level density of the same family of $L$-functions.


\section{Statement of Results}

We use the standard notation $s=\sigma+it$ for the real and imaginary parts of the complex variable $s$. As is standard in analytic number theory, we use the variable $\varepsilon$ to represent an arbitrarily small positive real number that may not be the same in each instance. Throughout the paper, $v$ denotes a fixed integer that is either $1$ or $3$. This allows us to prove results for the cases $p\equiv 1$ and $p\equiv 3$ (mod $4$) simultaneously. Also, let $\mbox{Li}(X) $ denote the logarithmic integral
\begin{equation}\label{Li}
\begin{split}
\mbox{Li}(X) \ = \ \int_2^X \frac{du}{\log u} \ = \ \frac{X}{\log X} \ + \ O\left(\frac{X}{\log^2 X}\right).
\end{split}
\end{equation}

To state our main result, we suppose that $K(s)$ is a function with the following properties: (i) it is analytic for $-1<\sigma<2$; (ii) $K(\frac{1}{2}+it)=K(\frac{1}{2}-it)$ for all real $t$; (iii) $K(\frac{1}{2}+it)$ has rapid decay as $|t|\rightarrow \infty$; (iv) $K(1/2)\neq 0$; and (v) $K(s) = o(\log^{-2}|t|)$ uniformly for $-1<\sigma<2$ as $|t|\rightarrow \infty$. We further assume that
\begin{equation}\label{Kbound2}
\int_{-\infty}^{\infty} |K(c+it)|\log(|t|+2) \,dt \ < \ \infty
\end{equation}
for $-1<c<2$. We suppose that the integral $a(y)$ defined by
\begin{equation}\label{aadef}
a(y) \ = \ \frac{1}{2\pi i} \int_{(c)} K(s) y^{-s}\,ds
\end{equation}
is absolutely convergent for $y>0$ and $-1<c<2$, and assume that
\begin{equation}\label{abound2}
\int_0^{\infty} |a'(u)|u  \,du  \ < \ \infty.
\end{equation}
These assumptions about $K(s)$ seem to be reasonable, and are satisfied by many functions, including $K(s)=\exp((s-\frac{1}{2})^2)$.

The low-lying zeros of the family of $L$-functions $L(s,\chi_p)$, where $\chi_p=(\cdot |p)$ is the Legendre symbol, are expected to display the same statistics as the eigenvalues of the matrices from $USp(2N)$, chosen with respect to the Haar measure. Our main result gives a formula for the average one-level density for the zeros of this family.
\begin{thm}\label{thm2}
Assume GRH and the assumptions for $K(s)$ given above. Let $r(\alpha)$ be a continuous even function in $L^2(-\infty,\infty)$ such that $\int_{-\infty}^{\infty} |\alpha r(\alpha)|\,d\alpha$ converges and $\hat{r}(\alpha)$ is supported on the interval $[-1+\varepsilon,1-\varepsilon]$. Then, for $v=1$ or $v=3$,

\begin{equation*}
\begin{split}
\left(\frac{1}{2}K\left( \frac{1}{2}\right)\mbox{Li}(X) \right)^{-1} & \sum_{\begin{subarray}{c} p\leq X \\ p\equiv v (\mbox{\scriptsize mod } 4)\end{subarray}} \sum_{\rho(p)} K(\rho)\,r\left(\frac{\gamma \log X}{2\pi}\right) \\
\ = \ & \int_{-\infty}^{\infty} \left( 1-\frac{\sin (2\pi \alpha)}{2\pi\alpha}\right) r(\alpha)\,d\alpha \ + \ o(1)
\end{split}
\end{equation*}
as $X\rightarrow\infty$. Here, the sum $\sum_{\rho(p)}$ is over the nontrivial zeros $\rho=\frac{1}{2}+i\gamma$ of $L(s,\chi_p)$.
\end{thm}

Our proof of Theorem~\ref{thm2} is based on an asymptotic formula for the form factor $F(\alpha,X)$ defined by

\begin{equation}\label{Fdef}
F(\alpha,X) \ = \ \left(\frac{1}{4}K\left( \frac{1}{2}\right)\mbox{Li}(X) \right)^{-1} \sum_{\begin{subarray}{c} p\leq X \\ p\equiv v (\mbox{\scriptsize mod } 4)\end{subarray}} \sum_{\rho(p)} K(\rho) X^{i\alpha\gamma},
\end{equation}
where $\alpha$ is real, $X>1$, and the inner sum is over the nontrivial zeros $\rho=\frac{1}{2}+i\gamma$ of $L(s,\chi_p)$. Our asymptotic formula is as follows.

\begin{thm}\label{thm1}
Assume GRH. Suppose that $K(s)$ and $a(y)$ are as given above Theorem~\ref{thm2}. If $M>0$, $X\geq 2$, and $\alpha \in\mathbb{R}$, then

\begin{equation*}
\begin{split}
F(\alpha,X) \ = \ & -1 \ + \ \left( \frac{1}{2}K\left( \frac{1}{2}\right) \right)^{-1}X^{-\alpha/2}a\left( X^{-\alpha}\right)\left( \frac{X-\mbox{Li}(X)\log 2\pi}{\mbox{Li}(X)}\right) \\
& \ + \ O\left( X^{\frac{|\alpha|}{2}-\frac{1}{2}}\log^2 X\right) \ + \ O\left(X^{\left(-\frac{1}{4}+\varepsilon\right)|\alpha|}\right) \\
& \ + \ O\left(X^{-\frac{1}{2}+ \left(-\frac{3}{2}+\varepsilon\right)|\alpha|}\log^3 X\right) \ + \ O\left( \frac{X^{|\alpha|}\log^2 X}{X^{M/2}} \right),
\end{split}
\end{equation*}
with implied constants depending at most on $K(s)$, $M$, or $\varepsilon$.
\end{thm}

\begin{cor}
Assume GRH and suppose that $K(s)$ is as given above Theorem~\ref{thm2}. If $0<|\alpha|<1$ is fixed, then

$$
F(\alpha,X) \ = \ -1 \ + \ o(1)
$$
as $X\rightarrow \infty$.
\end{cor}

We use Theorem~\ref{thm1} to prove the following nonvanishing result for $L(\frac{1}{2},\chi_p)$.
\begin{thm}\label{conditionalnonvanishing}
Assume GRH. Let $v=1$ or $3$. Then $L(\frac{1}{2},\chi_p)=0$ for at most $25\%$ of the primes $p\equiv v$~(mod $4$).
\end{thm}

In Section~\ref{ratiossection}, we follow the heuristic techniques of Conrey and Snaith~\cite{CS} and develop the following ratios conjecture for the family of quadratic Dirichlet $L$-functions $L(s,\chi_{p})$.

\begin{con}
\label{con:1}
Let $v=1$ or $3$. If $-\frac{1}{4}<\mathfrak{R}\alpha<\frac{1}{4}$, $ \frac{1}{\log X}\ll\mathfrak{R}\beta<\frac{1}{4}$, and $\mathfrak{I}\alpha,\mathfrak{I}\beta\ll_{\epsilon}X^{1-\epsilon}$ for every $\epsilon>0$, then

\begin{align*}
&\sum_{\substack{p\leq X \\ p\equiv v(\bmod4)}}\frac{L(\tfrac{1}{2}+\alpha,\chi_{p})}{ L(\tfrac{1}{2}+\beta,\chi_{p})} \\
&=\sum_{\substack{p\leq X \\ p\equiv v(\bmod4)}} \left(\frac{ \zeta(1+2\alpha)}{ \zeta(1+\alpha+\beta)} + \left(\frac{p}{\pi}\right)^{-\alpha}\frac{ \Gamma(1/4-\alpha/2)}{\Gamma(1/4+\alpha/2)} \frac{\zeta(1-2\alpha)}{ \zeta(1-\alpha+\beta)}\right)\nonumber \\
&\hspace{0.3in} +O(X^{1/2+\epsilon}).\nonumber
\end{align*}
\end{con}

Conjecture~\ref{con:1} implies the following formula for the average one-level density of $L(s,\chi_{p})$, complete with lower order terms.

\begin{thm}
\label{thm:3.1}
Assume Conjecture \ref{con:1}. Let $f(z)$ be an even function that is holomorphic throughout the strip $|\mathfrak{I}z|<2$, is real on the real line, and satisfies the bound $f(x)\ll1/(1+x^{2})$ for all real $x$. Then

\begin{align*}
&\sum_{\substack{p\leq X \\ p\equiv v(\bmod4)}}\sum_{\rho (p)}f(\gamma) = \frac{1}{2\pi}\int_{-\infty}^{\infty}f(t) \Bigg(\log\frac{p}{\pi}+\frac{1}{2} \frac{\Gamma^{'}}{\Gamma}(1/4+it/2)\nonumber \\
&+\frac{1}{2} \frac{\Gamma^{'}}{\Gamma}(1/4-it/2)+2\Bigg(\frac{\zeta^{'}(1+2it)}{\zeta(1+2it)}-\left(\frac{p}{\pi}\right)^{-it} \frac{\Gamma(1/4-it/2)}{\Gamma(1/4+it/2)}\zeta(1-2it)\Bigg)\Bigg)dt\\
& \hspace{0.3in} +O(X^{1/2+\epsilon}).
\end{align*}
Here, the sum $\sum_{\rho(p)}$ is over the nontrivial zeros $\frac{1}{2}+i\gamma$ of $L(s,\chi_p)$.
\end{thm}

The following corollary of the above theorem states that if the ratios conjecture holds, then the Density Conjecture is true for this family of $L$-functions. That is, the analogue of \eqref{DensityConjecture} holds for any test function $f$, without restrictions on the support of $\hat{f}$. In other words, the ratios conjecture implies Katz-Sarnak density conjecture.

\begin{cor}
\label{thm:3.2}
Assume Conjecture \ref{con:1}. Let $f(z)$ be an even function that is holomorphic throughout the strip $|\mathfrak{I}z|<2$, is real on the real line, and satisfies the bound $f(x)\ll1/(1+x^{2})$ for all real $x$. Then

\begin{equation*}
\lim_{X\rightarrow\infty}\left( \frac{\mbox{Li}(X)}{2}\right)^{-1}\sum_{\substack{p\leq X \\ p\equiv v(\bmod4)}}\sum_{\rho (p)}f\left( \frac{\gamma\log X}{2\pi}\right)= \int_{-\infty}^{\infty}f(x)\left(1-\frac{\sin(2\pi x)}{2\pi x}\right)dx.
\end{equation*}
\end{cor}
Corollary~\ref{thm:3.2} is formally consistent with Theorem~\ref{thm2} if we take the function $f(z)$ to be $f(z)=K(\frac{1}{2})^{-1}K(\frac{1}{2}+2\pi iz/\log X)r(z)$. However, Theorem~\ref{thm2} assumes that the support of $\hat{r}(\alpha)$ is contained in $(-1,1)$, while no such restriction is needed for the function $f$ in Corollary~\ref{thm:3.2}.

\section{The Explicit Formula}

As in \"{O}zl\"{u}k and Snyder~\cite{ozluksnyder1}, the starting point of our proof of Theorem ~\ref{thm1} is an explicit formula involving the sum

\begin{equation*}
-\sum_{n=1}^{\infty} a\left( \frac{n}{x}\right) \Lambda(n) \chi_p(n),
\end{equation*}
where $x\geq 1$, $\chi_p(n)$ is the Legendre symbol $(n|p)$, $\Lambda(n)$ is the von Mangoldt function, and $a(y)$ is defined by \eqref{aadef}. To deduce an explicit formula, we insert the definition \eqref{aadef} with $c>1$ for $a(n/x)$ into the above sum. We interchange the order of summation and arrive at

\begin{equation*}
-\sum_{n=1}^{\infty} a\left( \frac{n}{x}\right) \Lambda(n) \chi_p(n) \ = \ \frac{1}{2\pi i} \int_{(c)} K(s)x^s \, \frac{L'}{L} (s,\chi_p)\,ds,
\end{equation*}
where $c>1$. It is a well-known fact that if $\chi$ is a primitive character mod $q$, then $(L'/L)(s,\chi) \ll \log^2(q|t|)$ uniformly for $-1\leq \sigma\leq 2$ along a sequence of values of $t$ going to infinity (see for example \S~19 of Davenport~\cite{Davenport}). By this fact and property (v) of $K(s)$ stated above Theorem~\ref{thm2}, we may use the residue theorem to move the line of integration to $\sigma=-\frac{1}{2}$, say. This leaves the residues $\sum_{\rho (p)}K(\rho) x^{\rho}$, where $\rho$ runs through all the nontrivial zeros of $L(s,\chi_p)$, and the possible residue $K(0)$ if $p\equiv 1$ (mod~$4$). We bound the resulting integral along $\sigma=-\frac{1}{2}$ by applying the functional equation for $L(s,\chi_p)$, Stirling's formula, and \eqref{Kbound2}. The result, after rearranging, is the explicit formula

\begin{align}\label{explicitformula}
\sum_{\rho(p)} K(\rho) x^{i\gamma}\ &= \ -x^{-1/2}\sum_{n=1}^{\infty} a\left( \frac{n}{x}\right) \Lambda(n) \chi_p(n)\nonumber \\
& \qquad + \ x^{-1/2}a\left( \frac{1}{x}\right)\log \frac{p}{2\pi} \ + \ O\left(x^{-{1/2}}\right)
\end{align}
for $x\geq 1$. Here, the implied constant depends only on $K(s)$.

To estimate $F(\alpha,X)$, we define $F_1(x,X)$ by

\begin{equation}\label{F1def}
F_1(x,X) \ = \ \sum_{\begin{subarray}{c} p\leq X \\ p\equiv v (\mbox{\scriptsize mod } 4)\end{subarray}} \sum_{\rho(p)} K(\rho) x^{i\gamma},
\end{equation}
for $x\geq 1$ and $X\geq 2$. We insert the explicit formula \eqref{explicitformula} into this definition to deduce that

\begin{equation}\label{F1def2}
\begin{split}
F_1(x,X) \ = \ A \ + \ B \ + \ C,
\end{split}
\end{equation}
where, since $\chi_p(n)=(n|p)$,

\begin{equation}\label{Adef}
\begin{split}
A \ = \ -x^{-1/2}\sum_{\begin{subarray}{c} p\leq X \\ p\equiv v (\mbox{\scriptsize mod } 4)\end{subarray}}\sum_{n=1}^{\infty} a\left( \frac{n}{x}\right) \Lambda(n) \left( \frac{n}{p}\right) ,
\end{split}
\end{equation}

\begin{equation}\label{Bdef}
\begin{split}
B \ = \ x^{-1/2}a\left( \frac{1}{x}\right)\sum_{\begin{subarray}{c} p\leq X \\ p\equiv v (\mbox{\scriptsize mod } 4)\end{subarray}} \log\frac{p}{2\pi} ,
\end{split}
\end{equation}
and

\begin{equation}\label{Cdef}
\begin{split}
C \ \ll \ x^{-1/2}\sum_{\begin{subarray}{c} p\leq X \\ p\equiv v (\mbox{\scriptsize mod } 4)\end{subarray}} 1 \
\end{split}
\end{equation}
for $x\geq 1$ and $X\geq 2$, say.

To estimate $A$, we split the $n$-sum in \eqref{Adef} into squares and non-squares to deduce that

\begin{equation}\label{Adef2}
\begin{split}
A \ = \ A_1 \ + \ A_2,
\end{split}
\end{equation}
where

\begin{equation}\label{A1def}
\begin{split}
A_1 \ = \ -x^{-1/2}\sum_{\begin{subarray}{c} p\leq X \\ p\equiv v (\mbox{\scriptsize mod } 4)\end{subarray}}\sum_{\begin{subarray}{c} n=1 \\ n=\Box \end{subarray}}^{\infty} a\left( \frac{n}{x}\right) \Lambda(n) \left( \frac{n}{p}\right)
\end{split}
\end{equation}
and

\begin{equation}\label{A2def}
\begin{split}
A_2 \ = \ -x^{-1/2}\sum_{\begin{subarray}{c} p\leq X \\ p\equiv v (\mbox{\scriptsize mod } 4)\end{subarray}}\sum_{\begin{subarray}{c} n=1 \\ n\neq\Box \end{subarray}}^{\infty} a\left( \frac{n}{x}\right) \Lambda(n) \left( \frac{n}{p}\right).
\end{split}
\end{equation}
Here, we write $n=\Box$ to indicate that the sum is over perfect squares $n$, and we use $n\neq \Box$ to denote a sum that is over integers that are not perfect squares.

\section{The Diagonal Terms}
We first consider the $A_1$, i.e., equation \eqref{A1def}. If $n$ is a square, then $(n|p)=1$ if $p\nmid n$ and $(n|p)=0$ if $p| n$. Therefore we may write

\begin{equation*}
\begin{split}
A_1 \ &= \ -x^{-1/2}\sum_{\begin{subarray}{c} p\leq X \\ p\equiv v (\mbox{\scriptsize mod } 4)\end{subarray}}\sum_{\begin{subarray}{c} n=1 \\ n=\Box \\ p\nmid n \end{subarray}}^{\infty} a\left( \frac{n}{x}\right) \Lambda(n) \\
& \ = \ -x^{-1/2}\sum_{\begin{subarray}{c} p\leq X \\ p\equiv v (\mbox{\scriptsize mod } 4)\end{subarray}}\sum_{\begin{subarray}{c} m=1 \\ p\nmid m \end{subarray}}^{\infty} a\left( \frac{m^2}{x}\right) \Lambda(m)
\end{split}
\end{equation*}
because $\Lambda(m^2)=\Lambda(m)$ for all positive integers $m$. We insert into this the definition \eqref{aadef} with $c>1$ and interchange the order of summation to deduce that

\begin{equation*}
\begin{split}
A_1 \ = \ -x^{-1/2}\sum_{\begin{subarray}{c} p\leq X \\ p\equiv v (\mbox{\scriptsize mod } 4)\end{subarray}} \ \frac{1}{2\pi i} \int_{(c)} K(s)x^s \, \left( -\frac{\zeta'}{\zeta} (2s) \ - \ \frac{\log p}{p^{2s}-1}\right)\,ds.
\end{split}
\end{equation*}
It is a well-known fact that $(\zeta'/\zeta)(s) \ll \log^2|t|$ uniformly for $-1\leq \sigma\leq 2$ along a sequence of values of $t$ going to infinity (see for example \S~17 of Davenport~\cite{Davenport}). It follows from this and property (v) of $K(s)$, stated above Theorem~\ref{thm2}, that we may apply the residue theorem and move the line of integration to $\sigma=\frac{1}{4}+\varepsilon$, leaving a residue from the pole of $\frac{\zeta'}{\zeta} (2s)$. To bound the resulting integral along $\sigma=\frac{1}{4}+\varepsilon$, we use the assumption \eqref{Kbound2}, the bound $(\zeta'/\zeta)(2s) \ll \log|t|$ for $\sigma=\frac{1}{4}+\varepsilon$ (see for example \S~15 of Davenport~\cite{Davenport}), and the fact that $\frac{\log p}{p^{2s}-1}\ll 1$ for $\sigma\geq \frac{1}{4}+\varepsilon$ and arrive at

\begin{equation*}
\begin{split}
A_1 \ = \ -x^{-1/2}\sum_{\begin{subarray}{c} p\leq X \\ p\equiv v (\mbox{\scriptsize mod } 4)\end{subarray}} \left( \frac{1}{2}K\left( \frac{1}{2}\right) x^{1/2} \ + \ O\left( x^{\frac{1}{4}+\varepsilon}\right) \right)
\end{split}
\end{equation*}
for $x\geq 1$, with implied constant depending only on $K(s)$. This simplifies to

\begin{equation}\label{A11def2}
\begin{split}
A_1 \ = \ -\pi(X;4,v)\, \left( \frac{1}{2}K\left( \frac{1}{2}\right)  \ + \ O\left( x^{-\frac{1}{4}+\varepsilon}\right) \right)
\end{split}
\end{equation}
for $x\geq 1$, where $\pi(X;4,1)$ denotes the number of primes $p\leq X$ that satisfy $p\equiv v$ (mod $4$). On GRH, the prime number theorem for arithmetic progressions implies (see, for example, \S~20 of Davenport~\cite{Davenport})

\begin{equation*}
\begin{split}
\pi(X;4,v) \ = \ \frac{1}{2}\mbox{Li}(X) \ + \ O\left( X^{1/2}\log X\right).
\end{split}
\end{equation*}
It follows from this and \eqref{A11def2} that

\begin{equation}\label{A1def3}
\begin{split}
A_1 \ = \ -\frac{1}{4}K\left( \frac{1}{2}\right)\mbox{Li}(X) \ + \ O\left( X^{1/2}\log X\right) \ + \ O\left( x^{-\frac{1}{4}+\varepsilon} \,\mbox{Li}(X)  \right)
\end{split}
\end{equation}
for $x\geq 1$ and $X\geq 2$, with implied constant depending only on $K(s)$.

\section{The Off-diagonal Terms}\label{offdiag}
To estimate $A_2$ defined by \eqref{A2def}, we first prove

\begin{lem}\label{abound}
Let $K(s)$ be a function having the properties stated above as in Theorem~\ref{thm2}, and let $a(y)$ be defined by \eqref{aadef}. If $-1<c<2$, then

$$
a(y) \ \ll \ y^{-c}
$$
for all $y>0$, where the implied constant depends only on $K(s)$ and $c$. Moreover, if $X>0$ then the function $\alpha\mapsto X^{-\alpha/2} a(X^{-\alpha})$ is an even function of the real variable $\alpha$.
\end{lem}

\begin{proof}
The first assertion follows immediately from putting absolute values inside the integral in the definition \eqref{aadef} of $a(y)$, since it converges absolutely by assumption.
To prove the second assertion, observe that the definition \eqref{aadef} of $a(y)$ with $c=\frac{1}{2}$ and property (ii) of $K(s)$ imply

\begin{equation*}
\begin{split}
a\left( \frac{1}{y}\right) \ &= \ \frac{1}{2\pi } \int_{-\infty}^{\infty} K(\tfrac{1}{2}+it) y^{\frac{1}{2}+it}\,dt \\
& \ = \ \frac{1}{2\pi } \int_{-\infty}^{\infty} K(\tfrac{1}{2}-iu) y^{\frac{1}{2}-iu}\,du \ = \ ya(y).
\end{split}
\end{equation*}
by a change of variable $u=-t$.
\end{proof}

The main tool we will use to estimate the contribution $A_2$ of the off-diagonal terms is the following consequence of GRH.

\begin{lem}\label{GRHbound}
Assume GRH, and let $v=1$ or $3$. If $M>0$ is fixed and $n\leq X^M$ is an integer that is not a perfect square, then

$$
\sum_{\begin{subarray}{c} p\leq X \\ p\equiv v (\mbox{\scriptsize mod } 4)\end{subarray}} \left( \frac{n}{p}\right) \ \ll \ X^{1/2}\log X
$$
for $X\geq2 $, with implied constant depending only on $M$.
\end{lem}

\begin{proof}
Using characters to detect congruence modulo $4$, we may write the sum in question as

$$
\frac{1}{2}\sum_{5\leq p\leq X } \left( \left(\frac{1}{p}\right)\ \pm \ \left( \frac{-1}{p} \right) \right)\left( \frac{n}{p}\right) \ = \ \frac{1}{2}\sum_{5\leq p\leq X } \left( \frac{n}{p}\right) \ \pm \ \frac{1}{2}\sum_{5\leq p\leq X } \left(\frac{-n}{p}\right),
$$
where the symbol $\pm$ is plus if $v=1$ and minus if $v=3$. Thus it suffices to prove that

\begin{equation}\label{GRHbound1}
\sum_{5\leq p\leq X } \left( \frac{n}{p}\right) \ \ll_M \ X^{1/2}\log X
\end{equation}
for integers $n$ with $|n|\leq X^M$ that are not perfect squares. For such $n$, the character $\left(\frac{n}{\cdot }\right)$ modulo $n$ is non-principal since if a number $N$ is a square modulo every prime then $N$ must be a square (see, for example, Theorem~1 of Hall~\cite{Hall}). It follows from this and the explicit formula for Dirichlet characters (see \S~20 of Davenport~\cite{Davenport}; the estimates there hold uniformly for $q\leq x^M$ for any fixed $M$) that, on GRH,

$$
\sum_{\nu\leq X} \left( \frac{n}{\nu}\right) \Lambda(\nu) \ \ll_M \ X^{1/2}\log^2 X
$$
uniformly for $n\leq X^M$ and $X\geq 2$. The prime powers with exponents greater than $2$ contribute at most $O(X^{1/2})$, and we thus arrive at

$$
\sum_{5\leq p\leq X} \left( \frac{n}{p}\right) \log p \ \ll_M \ X^{1/2}\log^2 X.
$$
It now follows from this and partial summation that

\begin{equation*}
\begin{split}
\sum_{5\leq p\leq X} \left( \frac{n}{p}\right) \ = \ & \frac{1}{\log X} \sum_{5\leq p\leq X} \left( \frac{n}{p}\right) \log p \ + \ \int_4^X \left( \sum_{5\leq p\leq u} \left( \frac{n}{p}\right) \log p\right) \frac{du}{u\log^2 u} \\
\ll \ & X^{1/2}\log X \ +\ \int_4^X  u^{-1/2}\,du \ \ll \ X^{1/2}\log X.
\end{split}
\end{equation*}
This proves \eqref{GRHbound1} and thus Lemma~\ref{GRHbound}.
\end{proof}

We now bound $A_2$. Let $M>0$ be arbitrarily large, and split the inner sum in the definition \eqref{A2def} of $A_2$ into the sum of the terms with $n\leq X^M$ and the sum of the rest to write

\begin{equation}\label{A2def2}
\begin{split}
A_2 \ = \ A_{21} \ + \ A_{22},
\end{split}
\end{equation}
where

\begin{equation}\label{A21def}
\begin{split}
A_{21} \ = \ -x^{-1/2}\sum_{\begin{subarray}{c} n\leq X^M \\ n\neq\Box \end{subarray}} a\left( \frac{n}{x}\right) \Lambda(n)\sum_{\begin{subarray}{c} p\leq X \\ p\equiv v (\mbox{\scriptsize mod } 4)\end{subarray}} \left( \frac{n}{p}\right)
\end{split}
\end{equation}
and

\begin{equation}\label{A22def}
\begin{split}
A_{22} \ = \ -x^{-1/2}\sum_{\begin{subarray}{c} n> X^M \\ n\neq\Box \end{subarray}} a\left( \frac{n}{x}\right) \Lambda(n)\sum_{\begin{subarray}{c} p\leq X \\ p\equiv v (\mbox{\scriptsize mod } 4)\end{subarray}} \left( \frac{n}{p}\right).
\end{split}
\end{equation}
To estimate $A_{21}$, we let

\begin{equation}\label{lambdadef}
\lambda(n) \ = \ \Lambda(n)\sum_{\begin{subarray}{c} p\leq X \\ p\equiv v (\mbox{\scriptsize mod } 4)\end{subarray}} \left( \frac{n}{p}\right)
\end{equation}
and write \eqref{A21def} as

\begin{equation*}
\begin{split}
A_{21} \ = \ -x^{-1/2}\sum_{\begin{subarray}{c} n\leq X^M \\ n\neq\Box \end{subarray}} a\left( \frac{n}{x}\right) \lambda(n).
\end{split}
\end{equation*}
From this and partial summation, we arrive at

\begin{equation}\label{A21def2}
\begin{split}
A_{21} \ = \ -x^{-1/2} a\left( \frac{X^M}{x}\right)\sum_{\begin{subarray}{c} n\leq X^M \\ n\neq\Box \end{subarray}} \lambda(n) \ - \ x^{-1/2}\int_1^{X^M} \left( \sum_{\begin{subarray}{c} n\leq u \\ n\neq\Box \end{subarray}} \lambda(n)\right) \, a'\left( \frac{u}{x}\right) \, \frac{du}{x}.
\end{split}
\end{equation}
Observe that Lemma~\ref{GRHbound}, \eqref{lambdadef}, and the prime number theorem imply

\begin{equation*}
\sum_{\begin{subarray}{c} n\leq u \\ n\neq\Box \end{subarray}} \lambda(n) \ \ll \ X^{1/2}\log X \, \sum_{n\leq u} \Lambda(n) \ \ll \ uX^{1/2}\log X
\end{equation*}
for $u\geq 1$. It follows from this and \eqref{A21def2} that

\begin{equation*}
\begin{split}
A_{21} \ & \ll \ x^{-1/2} \left|a\left( \frac{X^M}{x}\right)\right|X^{M+\frac{1}{2}}\log X \\
& \qquad \ + \ x^{-1/2}X^{1/2}\log X \ \int_1^{X^M} u \, \left|a'\left( \frac{u}{x}\right)\right| \, \frac{du}{x}.
\end{split}
\end{equation*}
We apply Lemma~\ref{abound} with $c=\frac{3}{2}$, say, to bound the first term on the right-hand side, while we make a change of variable and use \eqref{abound2} to bound the second term. The result is

\begin{equation}\label{A21def3}
\begin{split}
A_{21}
\ll \ & \frac{x X^{\frac{1}{2}}\log X}{X^{M/2}} \ +\ x^{\frac{1}{2}} X^{\frac{1}{2}}\log X.
\end{split}
\end{equation}
Next, to estimate $A_{22}$ defined by \eqref{A22def}, we apply the trivial bound for the $p$-sum and Lemma~\ref{abound} with $c=\frac{3}{2}$, say, to deduce that

\begin{equation*}
\begin{split}
A_{22} \ \ll \ x^{-1/2}X\sum_{n>X^M} \frac{x^{3/2}\log n}{n^{3/2}} \ \ll \ \frac{xX\log X}{X^{M/2}}.
\end{split}
\end{equation*}
From this, \eqref{A2def2}, and \eqref{A21def3}, we arrive at

\begin{equation}\label{A2def3}
\begin{split}
A_{2} \ \ll \ x^{\frac{1}{2}} X^{\frac{1}{2}}\log X \ + \  \frac{xX\log X}{X^{M/2}}
\end{split}
\end{equation}
uniformly for $x\geq 1$ and $X\geq 2$.

It now follows from \eqref{Adef2}, \eqref{A1def3}, and \eqref{A2def3} that

\begin{equation}\label{Adef3}
\begin{split}
A \ = \ & -\frac{1}{4}K\left( \frac{1}{2}\right)\mbox{Li}(X) \ + \ O\left( x^{\frac{1}{2}} X^{\frac{1}{2}}\log X\right) \\
& \ \ \ + \ O\left( x^{-\frac{1}{4}+\varepsilon} \,\mbox{Li}(X)  \right) \ + \ O\left( \frac{xX\log X}{X^{M/2}} \right)
\end{split}
\end{equation}
uniformly for $x\geq 1$ and $X\geq 2$, where $M>0$ is a fixed arbitrarily large real number.

\section{The Asymptotic Formula}
Having estimated the term $A$ in \eqref{F1def2}, we now bound the terms $B$ and $C$, and finish the proof of Theorem~\ref{thm1}. To estimate $B$ defined by \eqref{Bdef}, we apply the prime number theorem for arithmetic progressions to deduce that

\begin{equation*}
\begin{split}
B \ = \ \frac{x^{-1/2}}{2}a\left( \frac{1}{x}\right)\left( X-\mbox{Li}(X)\log 2\pi\right) \ + \ O\left(x^{-1/2}\left|a\left( \frac{1}{x}\right) \right|X^{1/2}\log^2X\right).
\end{split}
\end{equation*}
From this and Lemma~\ref{abound} with $c=-1+\varepsilon$, it follows that

\begin{equation}\label{Bdef2}
\begin{split}
B \ = \ \frac{x^{-1/2}}{2}a\left( \frac{1}{x}\right)\left( X-\mbox{Li}(X)\log 2\pi\right) \ + \ O\left(x^{-\frac{3}{2}+\varepsilon}X^{\frac{1}{2}}\log^2X\right)
\end{split}
\end{equation}
uniformly for $x\geq 1$ and $X\geq 2$. Next, to bound $C$ defined by \eqref{Cdef}, observe that the prime number theorem again implies

\begin{equation*}
C \ \ll \ x^{-1/2}\mbox{Li}(X).
\end{equation*}
From this, \eqref{F1def2}, \eqref{Adef3}, and \eqref{Bdef2}, we conclude that

\begin{equation*}
\begin{split}
F_1(x,X) \ = \ & -\frac{1}{4}K\left( \frac{1}{2}\right)\mbox{Li}(X) \ + \ \frac{x^{-1/2}}{2}a\left( \frac{1}{x}\right)\left( X-\mbox{Li}(X)\log 2\pi\right) \\
& \ + \ O\left( x^{\frac{1}{2}} X^{\frac{1}{2}}\log X\right) \ + \ O\left( x^{-\frac{1}{4}+\varepsilon} \,\mbox{Li}(X)  \right)\\
& \ + \ O\left(x^{-\frac{3}{2}+\varepsilon}X^{\frac{1}{2}}\log^2X\right)  \ + \  O\left( \frac{xX\log X}{X^{M/2}} \right)
\end{split}
\end{equation*}
uniformly for $x\geq 1$ and $X\geq 2$, where $M>0$ is a fixed arbitrarily large real number. We divide both sides of this equation by $-1$ times the first term on the right-hand side, take $x=X^{\alpha}$, and use the definitions \eqref{F1def} of $F_1(x,X)$ and \eqref{Fdef} of $F(\alpha,X)$ to arrive at Theorem~\ref{thm1} for $\alpha\geq 0$. Theorem~\ref{thm1} for $\alpha\leq 0$ follows from the asymptotic formula \eqref{Li}, the second assertion of Lemma~\ref{abound}, and the fact that $F(\alpha,X)=F(-\alpha,X)$. This latter equation is true because of \eqref{Fdef}, the assumption that $K(\frac{1}{2}+it)=K(\frac{1}{2}-it)$ for real $t$, and the fact that the zeros of $L(s,\chi_p)$ are symmetric about the real axis.

\section{One-level Density}

In this section, we use the asymptotic formula in Theorem~\ref{thm1} to prove Theorem~\ref{thm2}. We take $M=2$, say, in Theorem~\ref{thm1} and arrive at

\begin{equation*}
\begin{split}
F(\alpha,X) \ = \ & -1 \ + \ \left( \frac{1}{2}K\left( \frac{1}{2}\right) \right)^{-1}X^{-|\alpha|/2}a\left( X^{-|\alpha|}\right)\left( \frac{X-\mbox{Li}(X)\log 2\pi}{\mbox{Li}(X)}\right) \\
& \ + \ O\left(X^{\left(-\frac{1}{4}+\varepsilon\right)|\alpha|}\right) \ + \ o(1)
\end{split}
\end{equation*}
as $X\rightarrow \infty$, uniformly for $\alpha\in [-1+\varepsilon,1-\varepsilon]$. We multiply the left-hand side by $\hat{r}(\alpha)$ and integrate over $(-\infty,\infty)$ to deduce that

\begin{equation}\label{rhat1}
\begin{split}
& \int_{-\infty}^{\infty} F(\alpha,X)\,\hat{r}(\alpha)\,d\alpha \ = \  -\int_{-\infty}^{\infty} \hat{r}(\alpha)\,d\alpha \ + \ O\left( \int_{-\infty}^{\infty}X^{\left(-\frac{1}{4}+\varepsilon\right)|\alpha|} \,d\alpha\right) \ + \ o(1) \\
& + \ \left( \frac{1}{2}K\left( \frac{1}{2}\right) \right)^{-1}\left( \frac{X-\mbox{Li}(X)\log 2\pi}{\mbox{Li}(X)}\right) \int_{-\infty}^{\infty}X^{-\alpha/2}a\left( X^{-\alpha}\right)\,\hat{r}(\alpha)\,d\alpha
\end{split}
\end{equation}
as $X\rightarrow \infty$, where we used the fact that $|\hat{r}(\alpha)|\leq \int |r|<\infty$. Note that

\begin{equation}\label{rhat2}
\int_{-\infty}^{\infty}X^{\left(-\frac{1}{4}+\varepsilon\right)|\alpha|} \,d\alpha \ = \ \frac{2}{\left(\frac{1}{4}-\varepsilon\right)\log X} \ = \ o(1)
\end{equation}
as $X\rightarrow \infty$. Also, the definition of $\hat{r}(\alpha)$ and the fact that its support is on $(-1,1)$ imply

\begin{equation}\label{rhat3}
\begin{split}
\int_{-\infty}^{\infty} \hat{r}(\alpha)\,d\alpha \ = \ \int_{-1}^{1} \hat{r}(\alpha)\,d\alpha \ = \ & \int_{-1}^1 \int_{-\infty}^{\infty} r(\beta)e^{-2\pi i \alpha\beta}\,d\beta\,d\alpha \\
\ = \ & 2 \int_{-\infty}^{\infty} r(\alpha) \left( \frac{\sin 2\pi \alpha}{2\pi \alpha}\right)\,d\alpha
\end{split}
\end{equation}
upon interchanging the order of integration and relabeling $\beta$ as $\alpha$. To estimate the last term on the right-hand side of \eqref{rhat1}, we use the Plancherel theorem to write, as in the proof of Theorem~3 of \"{O}zl\"{u}k and Snyder~\cite{ozluksnyder1},

\begin{equation*}
\begin{split}
\int_{-\infty}^{\infty}X^{-\alpha/2}a\left( X^{-\alpha}\right)\,\hat{r}(\alpha)\,d\alpha \ = \ \frac{K(\frac{1}{2})}{\log X} \int_{-\infty}^{\infty}r(\alpha)\,d\alpha \ + \ O\left(\frac{1}{\log^2X}\right)
\end{split}
\end{equation*}
uniformly for $X\geq 2$. The application of the Plancherel theorem is valid because we are assuming that $r\in L^2(-\infty,\infty)$. From this and \eqref{Li}, we arrive at

\begin{equation*}
\begin{split}
\left( \frac{1}{2}K\left( \frac{1}{2}\right) \right)^{-1}\left( \frac{X-\mbox{Li}(X)\log 2\pi}{\mbox{Li}(X)}\right) & \int_{-\infty}^{\infty}X^{-\alpha/2}a\left( X^{-\alpha}\right)\,\hat{r}(\alpha)\,d\alpha \\
\ = \ & 2\int_{-\infty}^{\infty}r(\alpha)\,d\alpha \ + \ O\left(\frac{1}{\log X}\right).
\end{split}
\end{equation*}
It follows from this, \eqref{rhat1}, \eqref{rhat2}, and \eqref{rhat3} that

\begin{equation}\label{rhat4}
\begin{split}
\int_{-\infty}^{\infty} F(\alpha,X)\,\hat{r}(\alpha)\,d\alpha \ = \ 2\int_{-\infty}^{\infty}\left( 1- \frac{\sin 2\pi \alpha}{2\pi \alpha}\right) r(\alpha)\,d\alpha \ + \ o(1)
\end{split}
\end{equation}
as $X\rightarrow \infty$. On the other hand, we insert the definition \eqref{Fdef} of $F(\alpha, X)$ into the integral on the left-hand side and interchange the order of summation to deduce that

\begin{equation*}
\begin{split}
& \int_{-\infty}^{\infty} F(\alpha,X)\,\hat{r}(\alpha)\,d\alpha \\
& \ \ \ \ \ \ = \ \left(\frac{1}{4}K\left( \frac{1}{2}\right)\mbox{Li}(X) \right)^{-1} \sum_{\begin{subarray}{c} p\leq X \\ p\equiv v (\mbox{\scriptsize mod } 4)\end{subarray}} \sum_{\rho(p)} K(\rho) \int_{-\infty}^{\infty} X^{i\alpha\gamma} \,\hat{r}(\alpha)\,d\alpha.
\end{split}
\end{equation*}
This implies

\begin{equation}\label{rhat5}
\int_{-\infty}^{\infty} F(\alpha,X)\,\hat{r}(\alpha)\,d\alpha  \ = \ \left(\frac{1}{4}K\left( \frac{1}{2}\right)\mbox{Li}(X) \right)^{-1} \sum_{\begin{subarray}{c} p\leq X \\ p\equiv v (\mbox{\scriptsize mod } 4)\end{subarray}} \sum_{\rho(p)} K(\rho)\,r\left(\frac{\gamma \log X}{2\pi}\right)
\end{equation}
because

\begin{equation*}
\int_{-\infty}^{\infty} X^{i\alpha\gamma} \,\hat{r}(\alpha)\,d\alpha \ = \ \hat{\hat{r}}\left(-\frac{\gamma \log X}{2\pi} \right) \ = \ r\left(\frac{\gamma \log X}{2\pi}\right).
\end{equation*}
Theorem~\ref{thm2} now follows from \eqref{rhat4} and \eqref{rhat5}.

\section{Nonvanishing at the Central Point}
To prove Theorem~\ref{conditionalnonvanishing}, we follow the argument in the proof of Corollary~3 of \"{O}zl\"{u}k and Snyder~\cite{ozluksnyder2}. Take $K(s)=\exp((s-\frac{1}{2})^2)$, and let $m_p$ denote the multiplicity of the point $1/2$ as a zero of $L(s,\chi_p)$, with $m_p=0$ if $L(\frac{1}{2},\chi_p)\neq 0$. Moreover, let $\lambda<1$ be close to $1$, and let $r(u)=(\sin (\pi \lambda u)/(\pi \lambda u))^2$, so that $\hat{r}(\alpha)=\lambda^{-2}\max\{\lambda-|\alpha|,0\}$. Since $K(\frac{1}{2}+it)>0$, $r(u)\geq 0$, and $r(0)=1$, it follows that

\begin{align*}
\left(\frac{1}{4}\mbox{Li}(X)\right)^{-1} & \sum_{\begin{subarray}{c} p\leq X \\ p\equiv v (\mbox{\scriptsize mod } 4)\end{subarray}}m_p \ = \ \left(\frac{1}{4}\mbox{Li}(X)\right)^{-1} \sum_{\begin{subarray}{c} p\leq X \\ p\equiv v (\mbox{\scriptsize mod } 4)\end{subarray}} \frac{1}{K(\frac{1}{2})} \sum_{\begin{subarray} \rho(p) \\ \rho=\frac{1}{2}\end{subarray}}K(\rho) \\
\ & \leq \ \left(\frac{1}{4}K(\frac{1}{2})\mbox{Li}(X)\right)^{-1} \sum_{\begin{subarray}{c} p\leq X \\ p\equiv v (\mbox{\scriptsize mod } 4)\end{subarray}}  \sum_{\rho(p)}K(\rho) r\left(\frac{\gamma\log X}{2\pi }\right).
\end{align*}
From this and \eqref{rhat5}, we deduce that

\begin{equation}\label{conditionalnonvanishing2}
\left(\frac{1}{4}\mbox{Li}(X)\right)^{-1} \sum_{\begin{subarray}{c} p\leq X \\ p\equiv v (\mbox{\scriptsize mod } 4)\end{subarray}}m_p \ \leq \ \int_{-\infty}^{\infty} F(\alpha,X) \hat{r}(\alpha)\,d\alpha.
\end{equation}
We next estimate the integral by evaluating each term on the right-hand side of \eqref{rhat1}. To evaluate the integral involving $a(X^{-\alpha})$, observe that the Plancherel theorem gives

\begin{equation*}
\int_{-\infty}^{\infty}X^{-\alpha/2}a\left( X^{-\alpha}\right)\,\hat{r}(\alpha)\,d\alpha \ = \ \int_{-\infty}^{\infty}\widehat{X^{-\alpha/2}a\left( X^{-\alpha}\right)}\,r(\alpha)\,d\alpha.
\end{equation*}
We write the definition of the Fourier transform of $X^{-\alpha/2}a\left( X^{-\alpha}\right)$ and make a change of variable to arrive at

\begin{equation*}
\begin{split}
& \int_{-\infty}^{\infty}X^{-\alpha/2}a\left( X^{-\alpha}\right)\,\hat{r}(\alpha)\,d\alpha \\
& \ = \ \frac{1}{\log X} \int_{-\infty}^{\infty}\int_0^{\infty}t^{1/2} a(t) e^{2\pi i \alpha \log t/\log X} \,\frac{dt}{t}  \,r(\alpha)\,d\alpha.
\end{split}
\end{equation*}
Interchanging the order of integration and using the definition of $\hat{r}(\alpha)$, we deduce that

\begin{equation*}
\int_{-\infty}^{\infty}X^{-\alpha/2}a\left( X^{-\alpha}\right)\,\hat{r}(\alpha)\,d\alpha \ = \ \frac{1}{\lambda^2\log X} \int_{X^{-\lambda}}^{X^{\lambda}}t^{1/2} a(t)\left( \lambda- \left| \frac{\log t}{\log X}\right|\right) \,\frac{dt}{t}.
\end{equation*}
The contribution of the term $\left| \frac{\log t}{\log X}\right|$ is $O((\lambda\log X)^{-2})$, and we may use Lemma~\ref{abound} to extend the interval of integration to $(0,\infty)$ with error $O(X^{-\lambda})$. It follows that

\begin{equation*}
\int_{-\infty}^{\infty}X^{-\alpha/2}a\left( X^{-\alpha}\right)\,\hat{r}(\alpha)\,d\alpha \ = \ \frac{1}{\lambda\log X} \int_{0}^{\infty}t^{1/2} a(t)\,\frac{dt}{t} \ + \ O_{\lambda}\left(\frac{1}{\log^2 X}\right).
\end{equation*}
By Mellin inversion and the definition \eqref{aadef} of $a(t)$, the first term on the right-hand side equals $(\lambda\log X)^{-1}K(\frac{1}{2})$. It now follows that the last integral on the right-hand side of \eqref{rhat1} equals

\begin{equation*}
\left( \frac{1}{2}K\left( \frac{1}{2}\right) \right)^{-1}\left( \frac{X-\mbox{Li}(X)\log 2\pi}{\mbox{Li}(X)\log X}\right) \left(\frac{1}{\lambda}K\left( \frac{1}{2}\right) \ + \ O\left( \frac{1}{\log X}\right) \right) \ = \ \frac{2}{\lambda} \ + \ o(1)
\end{equation*}
as $X\rightarrow \infty$, by the fact that Li$(X)\log X\sim X$. It follows from this, \eqref{rhat1}, \eqref{rhat2}, and $-\int_{-\infty}^{\infty}\hat{r}(\alpha)\,d\alpha =-r(0)=-1$ that

\begin{equation*}
\int_{-\infty}^{\infty} F(\alpha,X) \hat{r}(\alpha)\,d\alpha \ = \ -1 \ + \ \frac{2}{\lambda} \ + \ o(1).
\end{equation*}
From this and \eqref{conditionalnonvanishing2}, we arrive at

\begin{equation*}
\left(\frac{1}{4}\mbox{Li}(X)\right)^{-1} \sum_{\begin{subarray}{c} p\leq X \\ p\equiv v (\mbox{\scriptsize mod } 4)\end{subarray}}m_p \ \leq \ \ -1 \ + \ \frac{2}{\lambda} \ + \ o(1).
\end{equation*}
Now the functional equation for $L(s,\chi_p)$ implies that $m_p$ is always even. Therefore

\begin{equation*}
\left(\frac{1}{2}\mbox{Li}(X)\right)^{-1} \sum_{\begin{subarray}{c} p\leq X \\ p\equiv v (\mbox{\scriptsize mod } 4) \\ L(\frac{1}{2},\chi_p)=0\end{subarray}}1 \ \leq \ \ -\frac{1}{4} \ + \ \frac{1}{2\lambda} \ + \ o(1).
\end{equation*}
Since $\pi(X;4,v)\sim \frac{1}{2}\mbox{Li}(X)$, we arrive at Theorem~\ref{conditionalnonvanishing} by letting $\lambda<1$ tend to $1$.

\section{The Ratios Conjecture for $L(s,\chi_{p})$}\label{ratiossection}

In this section we use the ratios conjecture as originally developed by Conrey, Farmer and Zirnbauer in \cite{CFZ} to compute the one-level density function for zeros of the family of quadratic Dirichlet $L$-functions associated with the character $\chi_{p}$, complete with lower order terms. The ratios conjecture will provide a result consistent with Theorem~\ref{thm2}, but with no constraints on the support of the Fourier transform of the test function.

We follow closely the calculations of Conrey and Snaith \cite{CS}. For further details we recommend the reader to check the many applications of the ratios conjecture given in their paper.
Consider the sum of ratios

\begin{equation*}
R_P(\alpha,\beta) \ := \ \sum_{\substack{p\leq X \\ p\equiv v(\bmod4)}}\frac{L(\tfrac{1}{2}+\alpha,\chi_{p})}{L(\tfrac{1}{2} +\beta,\chi_{p})}.
\end{equation*}
%
%
As part of the ratios conjecture as presented in \cite{CFZ, CS}, we replace the $L(s,\chi_{p})$ in the numerator by the approximate functional equation

\begin{equation}
L(\tfrac{1}{2}+\alpha,\chi_{p})=\sum_{m<x}\frac{\chi_{p}(m)}{m^{1/2+\alpha}} +\left(\frac{p}{\pi}\right)^{-\alpha} \frac{\Gamma(1/4-\alpha/2)}{\Gamma(1/4+\alpha/2)} \sum_{n<y}\frac{\chi_{p}(n)}{n^{1/2-\alpha}}+\ \text{Remainder},\nonumber
\end{equation}
where $xy=p/(2\pi)$, and we replace the $L(s,\chi_{p})$ in the denominator by their infinite series

\begin{equation}
\frac{1}{L(s,\chi_{p})}=\sum_{h=1}^{\infty}\frac{\mu(h)\chi_{p}(h)}{h^{s}}.\nonumber
\end{equation}
We consider each of the 2 pieces separately and average term-by-term within those pieces. We only retain the terms where we are averaging over squares; in other words we use the main part of the following orthogonality relation for quadratic characters over primes:

\begin{equation}
\sum_{\substack{p\leq X \\ p\equiv v(\bmod4)}} \chi_p(n) =
\left\{
    \begin{array}{ll}
        X^{*} + \ \text{small}  & \mbox{if } n \ \text{is a square} \\
        \text{small} & \mbox{if } n \ \text{is not a square},
    \end{array}
\right.\nonumber
\end{equation}
where, for brevity, we denote

\begin{equation}
X^{*}=\sum_{\substack{p\leq X \\ p\equiv v(\bmod4)}}1 \ \sim \ \frac{1}{2}\mbox{Li}(X), \nonumber
\end{equation}
where Li$(X)$ is defined by \eqref{Li}. We compute these diagonal terms and complete the sums by extending the ranges of the summation variables to infinity. We then calculate the Euler products and express these sums as ratios of zeta functions. The sum of these expressions, one for each piece of the approximate functional equation, will provide us with the required conjecture.

We first restrict our attention to the first piece of the approximate functional equation. In other words, we consider

\begin{equation}
\sum_{\substack{p\leq X \\ p\equiv v(\bmod4)}}\sum_{h,m}\frac{\mu(h) \chi_{p}(hm)}{h^{1/2+\beta}m^{1/2+\alpha}}. \nonumber
\end{equation}
We retain only the terms for which $hm$ is square to write this as

\begin{equation}
X^{*}\sum_{hm=\Box}\frac{\mu(h)}{h^{1/2+\beta}m^{1/2+\alpha}}. \nonumber
\end{equation}
We express this sum as an Euler product

\begin{equation}
\prod_{p}\sum_{\substack{h+m \\ \text{even}}}\frac{\mu(p^{h})}{p^{h(1/2+\beta)+m(1/2+\alpha)}}. \nonumber
\end{equation}
The effect of $\mu(p^{h})$ is to limit the choices for $h$ to $0$ or $1$. When $h=0$ we have

\begin{equation}
\sum_{\substack{m \\ \text{even}}}\frac{1}{p^{m(1/2+\alpha)}} =1+\sum_{m=1}^{\infty}\frac{1}{p^{m(1+2\alpha)}} =1+ \frac{1}{p^{1+2\alpha}}\frac{1}{\left(1-\frac{1}{p^{1+2\alpha}}\right)}, \nonumber
\end{equation}
and when $h=1$ we obtain a contribution of

\begin{equation}
\sum_{\substack{m \\ \text{odd}}}\frac{\mu(p)}{p^{1/2+\beta+m(1/2+\alpha)}}= -\frac{1}{p^{1+\alpha+\beta}}\frac{1}{ \left(1-\frac{1}{p^{1+2\alpha}}\right)}. \nonumber
\end{equation}
Thus, the Euler product simplifies to

\begin{eqnarray}
& &\prod_{p}\left(1+\frac{1}{p^{1+2\alpha}} \frac{1}{\left(1-\frac{1}{p^{1+2\alpha}}\right)}- \frac{1}{p^{1+\alpha+\beta}} \frac{1}{\left(1-\frac{1}{p^{1+2\alpha}}\right)}\right)\nonumber \\
&=&\frac{\zeta(1+2\alpha)}{\zeta(1+\alpha+\beta)} \prod_{p}\left(1+\frac{1}{p^{1+2\alpha}} \frac{1}{\left(1-\frac{1}{p^{1+2\alpha}}\right)}- \frac{1}{p^{1+\alpha+\beta}} \frac{1}{\left(1-\frac{1}{p^{1+2\alpha}}\right)}\right)\nonumber \\
&\times&\left(1-\frac{1}{p^{1+2\alpha}}\right) \left(1-\frac{1}{p^{1+\alpha+\beta}}\right)^{-1}.\nonumber \\
&=&\frac{\zeta(1+2\alpha)}{\zeta(1+\alpha+\beta)}.\nonumber
\end{eqnarray}

The other piece can be determined by recalling the functional equation

\begin{equation}
L(\tfrac{1}{2}+\alpha,\chi_{p})=\left(\frac{p}{\pi}\right)^{-\alpha} \frac{\Gamma(1/4-\alpha/2)}{\Gamma(1/4+\alpha/2)} L(\tfrac{1}{2}-\alpha,\chi_{p}).\nonumber
\end{equation}
We thus infer that Conjecture~\ref{con:1} is true.


For applications to the one-level density we note that

\begin{equation}
\sum_{\substack{p\leq X \\ p\equiv v(\bmod4)}}\frac{L^{'}(\tfrac{1}{2}+r,\chi_{p})}{ L(\tfrac{1}{2}+r,\chi_{p})}= \frac{d}{d\alpha}R_{P}(\alpha;\beta)\bigg|_{\alpha=\beta=r}.\nonumber
\end{equation}
Now

\begin{equation}
\frac{d}{d\alpha}\frac{\zeta(1+2\alpha)}{ \zeta(1+\alpha+\beta)} \bigg|_{ \alpha=\beta=r}=\frac{\zeta^{'}(1+2r)}{\zeta(1+2r)}\nonumber
\end{equation}
and

\begin{align*}
&\frac{d}{d\alpha}\left(\frac{p}{\pi}\right)^{-\alpha} \frac{\Gamma(1/4-\alpha/2)}{\Gamma(1/4+\alpha/2)} \frac{\zeta(1-2\alpha)}{\zeta(1-\alpha+\beta)}\bigg|_{ \alpha=\beta=r}\\
& \qquad =-\left(\frac{p}{\pi}\right)^{-r} \frac{\Gamma(1/4-r/2)}{\Gamma(1/4+r/2)}\zeta(1-2r).\nonumber
\end{align*}
Because of the uniformity in the parameters $\alpha$ and $\beta$, we may differentiate the conjectural formula with respect to these parameters and the results are valid in the same ranges and with the same error terms. Thus the ratios conjecture implies that the following holds.

\begin{thm}
\label{thm:3}
Assuming Conjecture \ref{con:1}, $1/\log X\ll \mathfrak{R}r<\tfrac{1}{4}$ and $\mathfrak{I}r\ll_{\epsilon}X^{1-\epsilon}$ we have

\begin{align}
& \sum_{\substack{p\leq X \\ p\equiv v(\bmod4)}}\frac{L^{'}(\tfrac{1}{2}+r,\chi_{p})}{ L(\tfrac{1}{2}+r,\chi_{p})}\nonumber \\
& \qquad= \sum_{\substack{p\leq X \\ p\equiv v(\bmod4)}}\left(\frac{\zeta^{'}(1+2r)}{ \zeta(1+2r)}-\left(\frac{p}{\pi}\right)^{-r} \frac{\Gamma(1/4-r/2)}{\Gamma(1/4+r/2)} \zeta(1-2r)\right)+O(X^{1/2+\epsilon}).\nonumber
\end{align}
\end{thm}

We now use Theorem \ref{thm:3} to compute the one-level density function for zeros of quadratic Dirichlet $L$-functions associated to $\chi_{p}$. For simplicity, we assume that $f(z)$ is holomorphic throughout the strip $|\mathfrak{I}z|<2$, real on the real line, and even, and we suppose that $f(x)\ll1/(1+x^{2})$ as $x\rightarrow\infty$.

We consider

\begin{equation}
S_{1}(f):=\sum_{\substack{p\leq X \\ p\equiv v(\bmod4)}}\sum_{\rho (p)}f(\gamma),\nonumber
\end{equation}
where inner sum is over the zeros $\rho=\frac{1}{2}+i\gamma$ of $L(s,\chi_{p})$. Recall that we are assuming the Generalized Riemann Hypothesis (GRH).

We have, by the argument principle, that

\begin{equation}
S_{1}(f)=\sum_{\substack{p\leq X \\ p\equiv v(\bmod4)}}\frac{1}{2\pi i}\left(\int_{(c)}-\int_{(1-c)}\right) \frac{L^{'}(s,\chi_{p})}{L(s,\chi_{p})}f(-i(s-1/2))ds,\nonumber
\end{equation}
where $(c)$ denotes a vertical line from $c-i\infty$ to $c+i\infty$ and $3/4>c>1/2+1/\log X$. The integral on the $c$-line is

\begin{equation}
\label{eq:3.4}
\frac{1}{2\pi}\int_{-\infty}^{\infty}f(t-i(c-1/2)) \sum_{\substack{p\leq X \\ p\equiv v(\bmod4)}} \frac{L^{'}(1/2+(c-1/2+it),\chi_{p})}{L(1/2+(c-1/2+it),\chi_{p})}dt.
\end{equation}
It follows from the GRH that on the path of integration $(c)$,

\begin{equation}
\label{eq:3.5}
\frac{L^{'}(s,\chi_{p})}{L(s,\chi_{p})}\ll\log^{2}(|s|p).
\end{equation}
For $|t|>X^{1-\epsilon}$, we estimate the integral using \eqref{eq:3.5} and the bound on $f(x)$, and the result is $X^{\epsilon}$. By the ratios conjecture, i.e., by Theorem \ref{thm:3}, if $|t|<X^{1-\epsilon}$ then the sum over $p$ in \eqref{eq:3.4} is

\begin{align}
\label{eq:3.6}
&\sum_{\substack{p\leq X \\ p\equiv v(\bmod4)}}\left(\frac{\zeta^{'}(1+2r)}{\zeta(1+2r)} -\left(\frac{p}{\pi}\right)^{-r}\frac{\Gamma(1/4-r/2)}{ \Gamma(1/4+r/2)}\zeta(1-2r)\right)\Bigg|_{r=c-1/2+it} \nonumber \\
& \qquad +O(X^{1/2+\epsilon}).\nonumber
\end{align}
Since this quantity is $\ll X^{1+\epsilon}$ for $|t|<X^{1-\epsilon}$ and $f(t)\ll 1/t^{2}$, we may extend the integration in $t$ to infinity. Finally, since the integrand is regular at $r=0$, we may move the path of integration to $c=1/2$ and so obtain

\begin{align}
& \frac{1}{2\pi}\int_{-\infty}^{\infty}f(t)\sum_{\substack{p\leq X \\ p\equiv v(\bmod4)}}\left(\frac{\zeta^{'}(1+2it)}{ \zeta(1+2it)} -\left(\frac{p}{\pi}\right)^{-it} \frac{\Gamma(1/4-it/2)}{\Gamma(1/4+it/2)}\zeta(1-2it)\right)dt\nonumber \\
& \qquad +O(X^{1/2+\epsilon})\nonumber
\end{align}
For the integral on the $(1-c)$-line, we change variables, letting $s\rightarrow1-s$, and we use the functional equation

\begin{equation}
\frac{L^{'}(1-s,\chi_{p})}{L(1-s,\chi_{p})}= \frac{X^{'}(s,\chi_{p})}{X(s,\chi_{p})}- \frac{L^{'}(s,\chi_{p})}{L(s,\chi_{p})}\nonumber
\end{equation}
where

\begin{equation}
\frac{X^{'}(s,\chi_{p})}{X(s,\chi_{p})}=-\log\frac{p}{\pi}- \frac{1}{2}\frac{\Gamma^{'}}{\Gamma}\left(\frac{1-s}{2}\right) -\frac{1}{2}\frac{\Gamma^{'}}{\Gamma}\left(\frac{s}{2}\right).\nonumber
\end{equation}
The contribution from the $L^{'}/L$ term is now exactly as before, since $f$ is even. Thus, we have proved Theorem ~\ref{thm:3.1}.




To prove Corollary~\ref{thm:3.2}, we define

$$f(t)=g\left(\frac{t\log X}{2\pi}\right).$$
Scaling the variable $t$ from Theorem \ref{thm:3.1} as $\tau=(t\log X)/(2\pi)$, we write

\begin{align*}
&\sum_{\substack{p\leq X \\ p\equiv v(\bmod4)}}\sum_{\rho (p)}g\left(\frac{\gamma\log X}{2\pi}\right)=\frac{1}{\log X}\int_{-\infty}^{\infty}g(\tau)\sum_{\substack{p\leq X \\ p\equiv v(\bmod4)}}\Bigg(\log\frac{p}{\pi}\\
&\qquad +\frac{1}{2} \frac{\Gamma^{'}}{\Gamma}\left(1/4+\frac{i\pi\tau}{\log X}\right)+\frac{1}{2}\frac{\Gamma^{'}}{\Gamma}\left(1/4-\frac{i\pi\tau}{\log X}\right)+2\Bigg(\frac{\zeta^{'}(1+\tfrac{4i\pi\tau}{\log X})}{\zeta(1+\tfrac{4i\pi\tau}{\log X})}\nonumber \\
& \qquad -e^{-(2\pi i\tau/\log X)\log(p/\pi)}\frac{\Gamma(1/4-\tfrac{i\pi\tau}{\log X})}{\Gamma(1/4+\tfrac{i\pi\tau}{\log X})}\zeta\left(1-\frac{4i\pi\tau}{\log X}\right)\Bigg)\Bigg)d\tau +O(X^{1/2+\epsilon}).\nonumber
\end{align*}
For large $X$, only the $\log(p/\pi)$ term, the $\zeta^{'}/\zeta$ term, and the final term in the integral contribute, yielding the asymptotic

\begin{align*}
&\sum_{\substack{p\leq X \\ p\equiv v(\bmod4)}}\sum_{\rho (p)}g\left(\frac{\gamma\log X}{2\pi}\right)\\
& \qquad \sim\frac{1}{\log X}\int_{-\infty}^{\infty}g(\tau)\left(X^{*}\log X-X^{*}\frac{\log X}{2\pi i\tau}+X^{*}\frac{e^{-2\pi i\tau}}{2\pi i\tau}\log X\right)d\tau.\nonumber
\end{align*}
However, since $g$ is an even function, the middle term above drops out and the last term can be duplicated with a change of sign of $\tau$, leaving

\begin{equation*}
\lim_{X\rightarrow\infty}\frac{1}{X^{*}}\sum_{\substack{p\leq X \\ p\equiv v(\bmod4)}}\sum_{\rho (p)}g\left( \frac{\gamma\log X}{2\pi}\right)= \int_{-\infty}^{\infty}g(\tau)\left(1+\frac{e^{-2\pi i\tau}}{4\pi i\tau}+\frac{e^{2\pi i\tau}}{-4\pi i\tau}\right)d\tau.
\end{equation*}
This is exactly the expected asymptotic formula, and proves Corollary~\ref{thm:3.2}.

\vspace{0.5in}

\noindent \textbf{Acknowledgment:} We would like to thank Professor Zeev Rudnick for his comments on a previous version of the manuscript. \\



{
}

\end{document}